\documentclass[12pt, reqno]{amsart}
\usepackage{amsmath}
\usepackage{amssymb}
\usepackage{epsfig}
\usepackage{graphicx}
\usepackage{color}
\definecolor{shadecolor}{gray}{0.875}
\usepackage{amscd}
\usepackage{comment}
\usepackage{adjustbox}
\usepackage{multirow}
\usepackage{tikz-cd}
\usepackage[all]{xy}

\numberwithin{equation}{section}

\usetikzlibrary{positioning}
\usepackage[colorlinks=false,urlbordercolor=white]{hyperref}
  
\tikzset{sgplattice/.style={inner sep=1pt,norm/.style={red!50!blue},char/.style={blue!50!black},
  lin/.style={black!50}},cnj/.style={black!50,yshift=-2.5pt,left=-1pt of #1,scale=0.5,fill=white}}

\DeclareFontFamily{U}{mathb}{\hyphenchar\font45}
\DeclareFontShape{U}{mathb}{m}{n}{
      <5> <6> <7> <8> <9> <10> gen * mathb
      <10.95> mathb10 <12> <14.4> <17.28> <20.74> <24.88> mathb12
      }{}
\DeclareSymbolFont{mathb}{U}{mathb}{m}{n}
\DeclareMathSymbol{\righttoleftarrow}{3}{mathb}{"FD}

\theoremstyle{plain}
\newtheorem{prop}{Proposition}[section]

\newtheorem{theo}[prop]{Theorem}

\theoremstyle{definition}

\newtheorem{exam}[prop]{Example}

\newcommand{\eqto}{\stackrel{\lower1.5pt\hbox{$\scriptstyle\sim\,$}}\to}
\newcommand{\eqdashto}{\stackrel{\lower1.5pt\hbox{$\scriptstyle\sim\,$}}\dashrightarrow}
\newcommand{\actsfromleft}{\mathrel{\reflectbox{$\righttoleftarrow$}}}
\newcommand{\actsfromright}{\righttoleftarrow}

\def\cO{{\mathcal O}}

\def\bA{{\mathbb A}}

\def\bP{{\mathbb P}}

\def\bZ{{\mathbb Z}}

\def\bC{{\mathbb C}}

\def\rH{{\mathrm H}}

\def\Pic{\mathrm{Pic}}

\def\Aut{\mathrm{Aut}}

\def\Burn{\mathrm{Burn}}

\def\lim{\mathrm{lim}}
\def\NS{\mathrm{NS}}

\def\IJ{\mathrm{IJ}}
\def\JJ{\mathrm{J}}

\def\Cr{\mathrm{Cr}}

\author{Andrew Kresch}
\address{
  Institut f\"ur Mathematik,
  Universit\"at Z\"urich,
  Winterthurerstrasse 190,
  CH-8057 Z\"urich, Switzerland
}
\email{andrew.kresch@math.uzh.ch}

\author{Sho Tanimoto}
\address{Graduate School of Mathematics\\ Nagoya University
\\ Nagoya, 464-8602, Japan}
\email{sho.tanimoto@math.nagoya-u.ac.jp}

\author{Yuri Tschinkel}
\address{Courant Institute\\
                New York University \\
                New York, NY 10012 \\
                USA }
\address{Simons Foundation\\
                 160 Fifth Av.\\ 
                 New York, NY 10010}                
\email{tschinkel@cims.nyu.edu}

\title[Intermediate Jacobians and Burnside invariants]{Intermediate Jacobians and Burnside invariants}

\begin{document}

\date{December 12, 2025}

\begin{abstract}
We propose new invariants in equivariant birational geometry, combining equivariant intermediate Jacobians and the Burnside formalism, for smooth rationally connected threefolds with actions of finite groups.
\end{abstract}

\maketitle

\section{Introduction}
\label{sect:intro}

This note is inspired by \cite{CKK}, which introduced a version of {\em atomic} birational invariants of \cite{KKP} into equivariant geometry. 

We recall the main problem in this area: to determine whether a given generically free regular action of a finite group $G$ on a smooth projective rational variety $X$ of dimension $n$, over $\bC$, is equivariantly birational to a linear, or projectively linear, action on $\bP^n$, i.e., to an action arising from a projectivization  $\bP(V)$ of a $(n+1)$-dimensional representation $V$ of $G$, respectively, of a central extension of $G$. We refer to \cite{HKT-small} and \cite{HT-odd} for an introduction to these notions. 
The linearization problem is settled in dimension 2 \cite{pinardin}, but is largely open in dimensions $\ge 3$. 

Here, we focus on threefolds. We 
connect the Burnside formalism of \cite{BnG} with the theory of (equivariant) intermediate Jacobians to recover the most striking applications in \cite{CKK} in a more classical framework. Concretely, the invariants we offer take into account only the stabilizer stratification and the $G$-action on the intermediate Jacobian. 

Our main contributions in this paper are:
\begin{itemize}
\item 
definition of new birational invariants of $G$-actions on rationally connected threefolds over $\bC$, for arbitrary finite $G$, in Section~\ref{sect:defn};
\item 
Proposition~\ref{prop:G}, a classical analog of \cite[Theorem 3.6]{CKK}; 
\item 
applications to conic bundles, quadric surface bundles, and nodal cubic threefolds, in Section~\ref{sect:appl}. 
\end{itemize}

\

\noindent
{\bf Acknowledgments:}
The second author was partially supported by JST FOREST program Grant number JPMJFR212Z, and by JSPS KAKENHI Grand-in-Aid (B) 23K25764.
The third author was partially supported by NSF grant 2301983.

\section{Generalities}
\label{sect:gen}

\subsection*{Notation}
Throughout, we work over $k=\bC$, the complex numbers. We let $G$ be a finite group. By convention, $G$-actions on varieties are from the right, we emphasize this by writing $X\actsfromright G$; correspondingly, the action on the function field is from the left, $G\actsfromleft k(X)$.  
We write 
$$
X\sim_G X'
$$ 
to indicate $G$-equivariant birationality of $X$ and $X'$. 

We briefly recall the framework of equivariant intermediate Jacobians as in \cite{CTT} and 
the main ingredients of the Burnside formalism developed in \cite{BnG} which are relevant for the construction of enhanced birational invariants. 

\subsection*{Curves and their Jacobians}

Let $C$ be an irreducible smooth projective curve of genus $g(C)\ge 1$
and $(\mathrm{J}(C),\theta_C)$ its Jacobian, with its
principal polarization. 
We have a homomorphism
\begin{equation}
\label{eqn.AutAut}
\Aut(C)\to \Aut((\mathrm{J}(C),\theta_C)). \end{equation}
It is well-known (see, e.g., \cite[Section 4]{Matsusaka}) that by \eqref{eqn.AutAut},
$$
\begin{cases} 
\Aut(C)/C(k)\cong \Aut((\mathrm{J}(C),\theta_C)), & \text{if } g(C)=1, \\ 
\Aut(C)\cong \Aut((\mathrm{J}(C),\theta_C)), & \text{if hyperelliptic}, g(C)\ge 2, \\
\Aut(C)\times\{\pm1\}\cong \Aut((\mathrm{J}(C),\theta_C)), & \text{otherwise}.
\end{cases}
$$

\subsection*{$G$-abelian varieties}
Let $(A,\theta_A)$ be a principally polarized abelian variety. It is called a $G$-equivariant principally  polarized abelian variety if $G$ acts regularly on $A$ preserving both the origin and the class of $\theta_A$ in the N\'eron-Severi group $\NS(A)$; the action is not assumed to be faithful. 

We will use the following observation \cite[Corollary 3.2]{CTT}:
A $G$-equivariant principally polarized abelian variety admits a unique, up to permutation of factors, decomposition as a product of indecomposable $G$-equivariant principally polarized abelian varieties.
In combination with \cite[Corollary 9.2]{Debarre}, we also see that in the non-equivariant decomposition of $A$ as a product of indecomposable principally polarized abelian varieties $A_\delta$, the union $\bigcup_\delta A_\delta$ is $G$-invariant, hence there is an induced $G$-action on the disjoint union $\sqcup_\delta A_\delta$.

\subsection*{Intermediate Jacobians}
Let $X$ be  a smooth projective rationally connected threefold and 
$$
\IJ(X):= \rH^3(X,\bC)/(\rH^1(X,\Omega^2_X)\oplus 
\rH^3(X,\bZ))
$$
its intermediate Jacobian, with its principal polarization
$\theta_X$
arising from the cup product
$$
\wedge^2  \rH^3(X,\bZ)\to \rH^6(X,\bZ)\simeq\bZ. 
$$
When $X$ is rational, $\IJ(X)$ is a product of Jacobians of curves. 

Examples of $X$ with computable intermediate Jacobians are standard conic bundles $\pi\colon X\to S$, with smooth discriminant curve $C\subset S$. There is an associated \'etale double cover $\tau\colon \widetilde{C}\to C$, parametrizing lines over $C$. The intermediate Jacobian $\IJ(X)$ is the Prym variety $\mathrm{P}(\tau)$ associated with $\tau$; it 
is the identity component of the locus in the Jacobian $\mathrm{J}(\widetilde{C})$ where the involution induced by $\tau$ acts as $(-1)$: 
$$
\IJ(X)=\mathrm{P}(\tau)=\mathrm{im}(1-\tau)=\bigl(\ker(1+\tau)\bigr)^0, 
$$
see, e.g., \cite{mumford} for more details regarding this construction.  
Let $G$ act on $X$.
Then $\IJ(X)$ is a $G$-equivariant principally polarized abelian variety.
We consider the non-equivariant decomposition
\begin{equation}
\label{eqn.Delta}
\IJ(X)=\prod_{\delta\in\Delta} \IJ_\delta(X)
\end{equation}
as a product of indecomposable principally polarized abelian varieties.
Then there is an induced $G$-action on $\Delta$, such that the orbits $\Delta/G$ index the
indecomposable $G$-equivariant principally polarized abelian varieties
$$
\IJ_{\omega}(X)=\prod_{\delta\in \omega}  \IJ_\delta(X),\qquad \omega\in \Delta/G.
$$ 
The $G$-action on $\IJ_{\omega}(X)$ induces an action on
$\sqcup_{\delta\in \omega}A_\delta$, transitive on components; in particular, the $A_\delta$, for $\delta\in \omega$, are non-equivariantly isomorphic.
For given $\delta\in \omega$, we get an action of the stabilizer $G_\delta$ on $\IJ_{\delta}(X)$, which we express as faithful action
\begin{equation}
\label{eqn.HdeltaGdelta}
\IJ_{\delta}\actsfromright G_\delta/H_\delta,
\qquad\text{with}\qquad 
H_\delta\subseteq G_\delta.
\end{equation}

Let $C$ be an irreducible smooth projective curve of positive genus and $\mathrm{J}(C)$ its Jacobian, with its principal polarization. 
We will say that $\omega$ is
{\em $C$-relevant} if $\IJ_\delta(X)\cong \mathrm{J}(C)$ for $\delta\in \omega$, as principally polarized abelian varieties.
Then, with the union of the $C$-relevant orbits, we have a $G$-invariant subset
\[
\Delta(C)\subseteq \Delta,
\]
such that the set of $C$-relevant orbits may be recovered as
\[
\{\text{$C$-relevant orbits}\}=\Delta(C)/G\subseteq \Delta/G.
\]

\subsection*{Burnside formalism}

The Burnside formalism takes as input for the analysis of a regular $G$-action on a smooth projective $X$ the following data: 
\begin{itemize}
    \item the stabilizer stratification, 
    \item representations of the stabilizers in the normal bundles of strata. 
\end{itemize}
A key initial step is passage to a birational model in \emph{divisorial form}.
This is a model satisfying the condition called
\emph{Assumption 2} in \cite[Section 3]{BnG}.
On such a model, the stabilizers are abelian, and their representations in the normal bundles decompose into direct sums of characters.  By \cite[Proposition 3.6]{BnG}, 
two birational models in divisorial form can be connected by a sequence of blow-ups and blow-downs with smooth centers, and each intermediate model is in divisorial form. 

In particular, this condition is satisfied if the action is in {\em standard form}, i.e., $X$ is 
smooth projective, with simple normal crossing boundary divisor
$$
D=\cup_{\alpha}D_\alpha,
$$
such that
\begin{itemize}
\item 
we have a free action of $G$ on $X\setminus D$, and
\item 
for all $\alpha$ and $g\in G$, either $g(D_\alpha)=D_{\alpha}$ or $g(D_\alpha)\cap D_\alpha=\emptyset$. 
\end{itemize}

The class of the action on an $n$-dimensional $X$, in divisorial form, is defined as a sum of symbols
\begin{equation}
    \label{eqn:symb} 
[X\actsfromright G] :=\sum_H\sum_F (H,Y\actsfromleft k(F),\beta),
\end{equation}
a sum over representatives $H$ of conjugacy classes 
of abelian subgroups of $G$
and over strata $F$ of dimension $d$ with abelian stabilizer $H$. The symbols record 
\begin{itemize}
    \item the 
residual action of a subgroup $Y\subseteq Z_G(H)/H$, the quotient of the centralizer of $H$ in $G$ by $H$, on the function field of $F$, and
\item a sequence $\beta=(b_1,\ldots, b_{n-d})$ of characters of $H$, which appear in the normal bundle to $F$. 
\end{itemize}
The expression takes values in a group
$$
\Burn_n(G)
$$
defined by symbols as in \eqref{eqn:symb}, subject to explicit relations \cite[Section 4]{BnG}. This group has an intricate internal structure. In particular, it admits a direct sum decomposition based on the birational class of the MRC quotient of the stratum $F$, by \cite[Remark 3.5]{KT-struct}. In the following section, we develop this framework in dimension 3, additionally taking into account information about the $G$-action on the intermediate Jacobian of the threefold. 

\section{Curve-localized Burnside groups}
\label{sect:defn}

We proceed with the definition of new invariants for $G$-actions on rationally connected threefolds combining intermediate Jacobians and Burnside invariants.

Let $C$ be an irreducible smooth projective curve of genus $\ge 1$. 
We define the \emph{$C$-localized Burnside group}
$$
\Burn_3^C(G)
$$
by generators and relations. The computation of the class of the $G$-action on a smooth projective rationally connected threefold $X$ in this group takes
into account only those strata in the stabilizer stratification and only those components of the intermediate Jacobian that are ``related'' to $C$, i.e., copies of $C$, ruled surfaces over $C$, and $\mathrm{J}(C)$.  

\subsection*{Generators}
The generators are symbols
\begin{align*}
& (H, Y \actsfromleft K, (b_1,b_2)), \\
& (H, Y \actsfromleft L, (b)), \\ 
& (H, \mathrm J \actsfromright Y),
\end{align*}
where, respectively,
\begin{itemize}
\item $H\subseteq G$ is an abelian subgroup, nontrivial characters $b_1,b_2$ generate the dual $H^\vee$, 
and $Y\subseteq Z_G(H)/H$ is a subgroup that acts faithfully on $K\cong k(C)$,
\item $H\subseteq G$ is nontrivial cyclic, with character group generated by $b$, and $Y\subseteq Z_G(H)/H$ acts faithfully on $L\cong k(C\times \bP^1)$, 
\item $H\subseteq G$, and
$Y\subseteq N_G(H)/H$ is a subgroup, with action on the principally polarized abelian variety $\mathrm{J}\cong \mathrm{J}(C)$ (action and isomorphism compatible with polarization), which
\begin{itemize}
\item is a faithful action, if $g(C)\ge 2$,
\item comes from a faithful action on $C$, if $g(C)=1$.
\end{itemize}
\end{itemize}
The symbols are subject to permutation of characters and conjugation relations as in \cite[Section 4]{BnG}:

{\bf (P)} (Permutation)
$$
(H, Y \actsfromleft K, (b_1,b_2)) = (H, Y \actsfromleft K, (b_2,b_1)), \quad \forall \, b_1,b_2.
$$

\ 

{\bf (C)} (Conjugation)
\begin{align*}
(H, Y \actsfromleft K, (b_1,b_2)) &=(H', Y' \actsfromleft K', (b_1',b_2')), 
\end{align*}
when there exists a $g\in G$ such that 
$$
H'=gHg^{-1},  \quad Y'=gYg^{-1}, 
$$
and $b_1',b_2'$ are $g$-conjugates of $b_1,b_2$;
and similarly for the other two kinds of symbols.

\

The blow-up relations of \cite{BnG} are modified to reflect the fact that a blow-up of a $G$-orbit of $C$
contributes one factor $\mathrm{J}(C)$ to the intermediate Jacobian for every component of the $G$-orbit, with $G$-action permuting the factors.

\

\noindent{\bf (B1)}:
For 
$
H, Y, b_1, b_2
$
as above, with $b_1+b_2 = 0$,
using the action of $Y$ on $\mathrm{J}(C)$ induced by $C \actsfromright Y$, we impose
\[
(H, Y\actsfromleft k(C), (b_1, -b_1)) + (H,   \mathrm{J}(C)\actsfromright Y) = 0.
\]

\noindent{\bf (B2)}:
For 
$
H, Y, b_1, b_2,
$
and action of $Y$ on $\mathrm{J}(C)$ as above,
$$
(H, Y\actsfromleft k(C), (b_1,b_2)) = \Theta_1+\Theta_2+
(H, \mathrm J(C)\actsfromright Y),
$$
where, with $\beta_1=(b_1,b_2-b_1)$, $\beta_2=(b_2,b_1-b_2)$,
\[
\Theta_1= \begin{cases} 0, & \text{if $b_1=b_2$}, \\
(H,Y \actsfromleft k(C), \beta_1)+(H,Y \actsfromleft k(C), \beta_2),
& \text{otherwise},
\end{cases}
\]
and
\[
\Theta_2= \begin{cases} 0, & \text{if $\langle b_1-b_2\rangle=H^\vee$}, \\
(\overline{H},\overline{Y} \actsfromleft k(C\times \bP^1), (b_1|_{\overline{H}})), & \text{otherwise}.
\end{cases}
\]
In the expression for $\Theta_2$ we put $\overline{H}=\ker(b_1-b_2)$ and apply the action construction, see \cite[Section 2]{BnG} or \cite[Section 2]{KT-struct}, 
to obtain $\overline{Y}$, with action on $k(C\times \bP^1)$. 

\ 
 
\noindent{\bf (B3)}:
For any $Y\subseteq G$ and $C\actsfromright Y$,
\[
(1, \mathrm{J}(C) \actsfromright Y) = 0
\]
for the corresponding action of $Y$ on $\mathrm{J}(C)$.

\ 
 
\noindent{\bf (B4)}:
For any $C$ with $g(C)=1$,
\[ (H, \mathrm{J}(C) \actsfromright Y)=(H_1,\mathrm{J}(C) \actsfromright Y/Y_1), \]
where $Y_1$ is the subgroup of $Y$, acting trivially on $\mathrm{J}(C)$, giving rise to $H$-extension $H_1$ of $Y_1$.

\subsection*{The class of the action}
Given a smooth projective rationally connected $G$-threefold in divisorial form, we define the class of the $G$-action in the $C$-localized Burnside group
$$
[X\actsfromright G]^C\in \Burn_3^C(G)
$$
as follows:
\begin{align*}
[X\actsfromright G]^C:=
 & \sum_F (H_F, Y_F\actsfromleft k(F), \beta_F(X)) + {}\\
& \sum_D (H_D, Y_D \actsfromleft k(D),\beta_D(X))  + \sum_\delta(H_\delta,\IJ_\delta(X)\actsfromright Y_{\delta}),
\end{align*}
where
\begin{itemize}
\item the first sum is over orbit representatives $F$ of subvarieties $F\cong C$ of $X$, where the generic stabilizer is $H_F$ and generic normal bundle representation $\beta_F(X)=(b_1,b_2)$ with nontrivial $b_1$, $b_2$,
\item the second sum is over orbit representatives $D$ of divisors in $X$ that are ruled surfaces over $C$, nontrivial generic stabilizer $H_D$, and generic normal bundle representation $\beta_D(X)=(b)$,
\item the third sum is over orbit representatives $\delta$ of $C$-relevant orbits $\omega\in \Delta/G$, with $\Delta$ as in \eqref{eqn.Delta}, where $H_\delta$ is as in \eqref{eqn.HdeltaGdelta},
\item in each sum there is the residual action of $Y_F$ on $k(F)$, respectively $Y_D$ on $k(D)$, respectively $Y_\delta$ on $\IJ_\delta(X)$.
\end{itemize}

\begin{exam} 
\label{exam:d2}
Let $X=S\times \bP^1$, where $S$ is a degree 2 del Pezzo surface.
The threefold $X$ has trivial intermediate Jacobian. Let $G=\bZ/2\bZ$,
acting by the standard covering involution of $\pi: S\to \bP^2$ 
and trivially on $\bP^1$. Let $C$ be the ramification curve of $\pi$, a smooth quadric curve. 
The $G$-fixed locus of $X$ is $C\times \bP^1$. 
The class of the $G$-action in the $C$-localized Burnside group is
$$
[X\actsfromright G]^C = (G,1\actsfromleft k(C\times \bP^1), (1)) \in \Burn_3^C(G).
$$
By relations \textbf{(B1)}--\textbf{(B2)} this is equal to 
$$
-2(G,\mathrm{J}(C)\actsfromright 1);
$$
see the
further discussion in Section \ref{sect:appl}.
\end{exam}

\section{Blow-up relations}
\label{sect:relations}

\begin{theo}
\label{thm:main}
Let $X$ and $X'$ be smooth projective rationally connected threefolds with a regular action of a finite group $G$, in divisorial form. Let $C$ be a curve of genus $\ge 1$. Then
$$
X\sim_G X'\quad \Rightarrow \quad 
[X\actsfromright G]^C=[X'\actsfromright G]^C.
$$
\end{theo}

\begin{proof}
By the consequence of functorial weak factorization, recalled in Section \ref{sect:gen},
the theorem reduces to the equality of classes in the $C$-localized Burnside group in the case of a $G$-equivariant blow-up
$$
\varrho \colon X'\to X, 
$$
with smooth center $Z$. Given the shape of the invariant, it suffices to consider the case that $Z$ is the orbit of a curve, isomorphic to $C$, which by abuse of notation we denote by $C$ in the analysis of the possibilities:
\begin{enumerate}
\item $C$ has trivial generic stabilizer,
\item the generic stabilizer $H$ of $C$ is nontrivial and $\beta_C(X)$ is of the form $(0,b)$,
\item the generic stabilizer $H$ of $C$ is nontrivial and $\beta_C(X)=(b_1,b_2)$ with nontrivial $b_1$, $b_2$.
\end{enumerate}

\noindent
{\bf Case (1)}: 
The generic stabilizer of the exceptional divisor $E$ of the blow-up is trivial, and there is no curve $C' \subset E$ which is isomorphic to $C$ and has nontrivial generic stabilizer. 
Thus there is no effect on the first two sums in the expression for $[X\actsfromright G]^C$.
The third sum gets an extra term, but this vanishes by \textbf{(B3)}.

\medskip
In case $g(C)=1$, it may be necessary to combine with \textbf{(B4)} to obtain the claimed vanishing.
In the remaining cases there is a similar, implicit use of \textbf{(B4)} when $g(C)=1$.

\medskip

\noindent{\bf Case (2)}:
In this case, there is a divisor $D$ of the stabilizer stratification such that its generic stabilizer is $H$, and $C\subset D$.
The first sum picks up a term $\beta_F(X)=(b,-b)$, and the third sum gets a contribution from the residual action on $C$.
Their sum vanishes by \textbf{(B1)}.

\medskip

\noindent
{\bf Case (3)}: 
Let $H$ be the generic stabilizer.
Let $E$ be the exceptional divisor. If $b_1 \ne b_2$, then the exceptional divisor $E$ admits two curves with stabilizer $H$, and respective weights $(b_1,b_2-b_1)$ and $(b_2,b_1-b_2)$.
If $\langle b_1-b_2\rangle$ is a proper subgroup of $H^\vee$, then $E$ has nontrivial generic stabilizer $\ker(b_1-b_2)$.
Thus the term $(H,Y\actsfromleft k(C),(b_1,b_2))$ in the first sum gets replaced by $\Theta_1$, with the addition to the second sum of $\Theta_2$, from \textbf{(B2)}.
The third term gets the required extra term, so that the equality $[X\actsfromright G]^C=[X'\actsfromright G]^C$ holds by \textbf{(B2)}.
\end{proof}

\begin{exam}
\label{exam:elliptic}
Let $X$ be a smooth rational threefold with 
a regular involution  $\iota$. Put $G=\langle \iota\rangle$. Let $C\subset X$ be an elliptic curve. We examine the 
blow-up relations:
\begin{itemize}
\item 
$C\subseteq X^G$, with normal bundle $(1,1)$. Blowing up we obtain an exceptional divisor birational to $C\times \bP^1$ with generic stabilizer $G$. We have the relation
$$
- (G, \mathrm{1}\actsfromleft k(C), (1,1)) + (G, \mathrm{1}\actsfromleft k(C\times \bP^1),(1)) + (G,\mathrm{J}(C)\actsfromright \mathrm{1}) =0.
$$
\item 
$C\subseteq X^G$, with normal bundle $(0,1)$, then 
$$
(G,\mathrm{1}\actsfromleft k(C),(1,1)) + (G,\mathrm{J}(C) \actsfromright \mathrm{1}) =0.
$$
\item 
$C$ has a $G$-action via translation by $\bZ/2$ and no stabilizer. After blowing up, we have the relation
$$
(G,\mathrm{J}(C)\actsfromright 1)=(1,\mathrm{J}(C)\stackrel{\text{triv}}{\actsfromright} G)=0.
$$
\item
$C$ has a $G$-action fixing 4 points.  
Then
$$
(1,\mathrm{J}(C)\actsfromright G)=0.
$$
\item $C$ has no $G$-action, no stabilizer.
Then
$$
(1,\mathrm{J}(C)\actsfromright 1)=0.
$$
\end{itemize}
In conclusion, all symbols involving $C$ vanish. 
\end{exam}

\section{Structure}
\label{sect:struct}

The paper \cite{KT-struct} introduced filtrations on the full Burnside group $\Burn_n(G)$, based on 
combinatorial properties of the subgroup lattice of $G$; these allow to simplify the analysis of the class $$
[X\actsfromright G]\in \Burn_n(G),
$$
in some cases. 

Briefly, \cite[Section 3]{KT-struct} introduced the notion of a {\em filter} $\mathbf{H}$, consisting of pairs $(H,Y)$, subject to certain properties, which ensure that the quotient 
$$
\Burn_n(G)\to \Burn_n^{\mathbf{H}}(G)
$$
by symbols with $(H,Y)\notin \mathbf{H}$ is a
well-defined homomorphism to a group that is generated by symbols with $(H,Y)\in \mathbf{H}$, with the same relations as in $\Burn_n(G)$, but applied only to these symbols.

Now let $C$ be an irreducible smooth projective curve of genus $\ge 2$.
Then, by the same reasoning, we have the $C$-localized Burnside group
$$
\Burn^C_3(G)\to \Burn_3^{\mathbf{H},C}(G),
$$
generated by symbols $(H,Y)\in \mathbf{H}$ and with relations \textbf{(B1)}--\textbf{(B3)}, applied only to these symbols.
(When $g=1$ the same reasoning is not applicable on account of the additional relation \textbf{(B4)}.)

For $G$ {\em abelian}, an example of a $G$-filter is
$$
\{(G,1)\},
$$
see \cite[Example 3.4]{KT-struct} and 
\cite[Section 8]{KT22}.
The corresponding Burnside group
\[ \Burn_n^G(G) \]
records only the strata with maximal stabilizer $G$, i.e., $G$-fixed loci. 
An analogous formalism applies to the $C$-localized Burnside group 
\[ \Burn_3^{G,C}(G). \]
We denote by $[X\actsfromright G]^{G,C}$ the image of the class $[X\actsfromright G]^G$, by means of the filter, in $\Burn_3^{G,C}(G)$.
Concretely, this is given by picking out just the symbols with first argument equal to $G$, in the formula for $[X\actsfromright G]^C$.

In particular, specializing to abelian groups $G$, 
we obtain:

\begin{prop}
\label{prop:toZ}
Let $G$ be abelian and $g(C)\ge 2$.
Then there is a homomorphism
\[ \varphi^G\colon \Burn_3^{G,C}(G)\to \bZ, \]
determined by
\begin{align*}
(G, 1 \actsfromleft K, (b_1,b_2)) &\mapsto -1, \\
(G, 1 \actsfromleft L, (b)) &\mapsto -2, \\ 
(G, \mathrm{J} \actsfromright 1) &\mapsto 1.
\end{align*}
\end{prop}

\begin{proof}
The abelian group $\Burn_3^{G,C}(G)$ is generated by symbols with $G$ as first argument, and has relations given by
\textbf{(B1)}--\textbf{(B3)} with $H=G$.
It is straightforward to verify that $\varphi^G$ respects the relations.
\end{proof}

This yields a more classical analog of \cite[Theorem 3.6]{CKK}:  

\begin{prop}
\label{prop:G}
Let $X$ be a smooth projective rationally connected threefold with a regular action of an abelian group $G$ and 
$$
X^G=\sqcup_{\alpha} F_{\alpha}
$$ 
the decomposition of the $G$-fixed locus 
into a disjoint union of smooth irreducible components.  
Let $C$ be an irreducible smooth projective curve of genus $\ge 2$ and 
\begin{itemize}
\item $I_1$ be the number of $F_{\alpha}$ isomorphic to $C$, 
\item $I_2$ be the number of $F_{\alpha}$ birational to $C\times \bP^1$, and 
\item $I_3$ be the number of factors of the intermediate Jacobian $\IJ(X)$ isomorphic to $\mathrm{J}(C)$, with trivial $G$-action. 
\end{itemize}
Then 
$$
I:=-I_1-2I_2+
I_3
$$
is a $G$-equivariant birational invariant, given by
\[ \varphi^G([X\actsfromright G]^{G,C}). \]
Furthermore, if $I\neq 0$ then the $G$-action on $X$ is not linearizable, and also not projectively linearizable. 
\end{prop}

\begin{proof}
The proof of the first statement is 
immediate from the definition of the class $[X\actsfromright G]^{G,C}$ and map $\varphi^G$.

To show the second statement, we pass to a standard model of the $G$-action on $\bP^3$, as in \cite{KT22}, 
and observe that the $G$-action cannot fix higher genus curves or nonrational surfaces. 
\end{proof}

\section{Applications}
\label{sect:appl}

In this section we provide several applications of the formalism of $C$-localized Burnside groups. The applications are most interesting when the group $G$ is small -- for large $G$, one can often deploy techniques from birational rigidity. For this reason, we focus on cyclic groups of order 2 and 3. 
For the treatment of the examples given here, the equivariant Burnside groups without $C$-localization are insufficient, on account of the relations for symbols in $\Burn_3(G)$. 

\subsection*{Involutions}

The first steps towards classification of involutions in the Cremona group $\Cr_3$ were undertaken by Prokhorov in  \cite{pro-inv}. Following the classification of involutions in $\Cr_2$, which is based on the existence of higher-genus curves in the fixed locus of the action, Prokhorov considered involutions $\iota$ on rational $X$ with a {\em nonuniruled} divisor 
in the fixed locus $X^\iota$. 
In \cite{CTT}, we constructed nonconjugated involutions in $\Cr_3$ without any divisors in the fixed locus; this was based on the intermediate Jacobian torsor obstruction, which already obstructs linearizability. 
Here, we offer further examples of involutions $\iota$ without nonuniruled divisors in $X^\iota$.

\begin{exam}
\label{exam:dp22}
We return to Example~\ref{exam:d2}: $X=S\times \bP^1$, where $S$ is a degree 2 del Pezzo surface and $G=\bZ/2\bZ$, acting via the covering involution on $S$, which fixes a smooth quartic curve $C$.
We have
$$
I_1=I_3=0, \quad I_2=1, 
$$
thus $I=-2$ in
Proposition~\ref{prop:G}.
In particular, the $G$-action on $X$ is not linearizable; see \cite[Theorem I]{CKK}.  However, by \cite[Theorem 1.1]{BP13}, 
we have 
$$
\rH^1(G, \Pic(S))=(\bZ/2\bZ)^6,
$$
so the $G$-action on $S$ is not even stably linearizable. 
\end{exam}

Let $G:=\langle \iota\rangle$ and  
$C$ be an irreducible smooth projective curve of genus $\ge 2$. 
We spell out 
generators and relations in $\Burn_3^C(G)$, taking into account that the only possibilities for a $C$-relevant
components of $\IJ(X)$ are
\begin{itemize}
    \item $\mathrm{J}(C)$ with trivial or nontrivial $G$-action,
    \item $\mathrm{J}(C)\times \mathrm{J}(C)$ with $G$-action permuting the factors. 
\end{itemize}
For the generators, we have
\begin{enumerate}
    \item $\mathfrak s_1:=(G, 1\actsfromleft k(C), (1,1))$,
    \item $\mathfrak s_2:=(G, 1\actsfromleft k(C\times \bP^1),(1))$,
    \item $\mathfrak s_3:=(G, \mathrm{J}(C)\actsfromright 1)$, 
    \item $\mathfrak s_4:=(1, \mathrm{J}(C)\actsfromright G)$, 
    \item $\mathfrak s_5:=(1, \mathrm{J}(C)\actsfromright 1)$.  
\end{enumerate}
For $C$ hyperelliptic, every involution on $\mathrm{J}(C)$ can be realized on $C$; thus,  
relation {\bf (B3)} implies that the symbols $\mathfrak s_4$ and $\mathfrak s_5$ 
vanish in $\Burn_3^C(G)$. Relations {\bf (B1)} and 
{\bf (B2)} yield
$$
\mathfrak s_1+\mathfrak s_3=0, \quad \mathfrak s_1=\mathfrak s_2+ \mathfrak s_3.
$$
It follows that 
$$
\Burn_3^C(G)\simeq \bZ
$$ 
and the homomorphism $\varphi^G$ of Proposition~\ref{prop:toZ} is an isomorphism; in particular, 
for $C$ hyperelliptic, the invariant $I$ of Proposition~\ref{prop:G} is the {\em only} invariant of involutions accessible via the $C$-localized Burnside groups.   

When $C$ is non-hyperelliptic and the $G$-action on $\mathrm{J}(C)$ does not arise from an automorphism of $C$, the corresponding symbol of type $\mathfrak s_4$ does not participate in blow-up relations.
In this case, $\Burn_3^C(G)$ is a free abelian group generated by $\mathfrak s_1$ and such symbols of type $\mathfrak s_4$.  

The following example gives an alternative approach to \cite[Example 6.9]{CTT}.

\begin{exam}
\label{exam:quadric}
Consider 
$$
X\subset \bP^1_{(t_1:t_2)}\times \bP^3_{(x_1:x_2:x_3:x_4)},
$$ 
given by the vanishing of 
$$
\sum_{i=0}^nt_1^it_2^{n-i} (f_i(x_1,x_2)+g_i(x_3,x_4)),  \quad  n\ge 3, 
$$
for general binary quadratic forms $f_i,g_i$, so that $X$ is smooth. 
Projection to $\bP^1$ yields a quadric surface bundle, with discriminant cover a smooth hyperelliptic curve $C$ of genus $g(C)=2n-1$, see \cite[Section 6]{CTT}. The threefold $X$ is rational, with $\IJ(X)=\mathrm{J}(C)$.  
The involution
$$
\iota\colon (x_1:x_2:x_3:x_4)\mapsto (-x_1:-x_2:x_3:x_4)
$$
fixes two (nonisomorphic, for general $f_i,g_i$) 
hyperelliptic curves $C',C''$ of genus $n-1$. Applying Proposition~\ref{prop:G}
to either $C'$  or $C''$ shows that $I\neq 0$ and the action is not linearizable.
\end{exam}

\begin{exam} 
\label{exam:incomp}
Assume that $X$ is rational with $\IJ(X)=\mathrm{J(C)}$, for 
a smooth projective non-hyperelliptic curve $C$ of genus $\ge 3$. 
Assume that 
the $G$-action on $\IJ(X)$ does not come from any $G$-action on $C$, i.e., some element of $G$ acts by an automorphism, not in the image of the homomorphism \eqref{eqn.AutAut}.
Then the $G$-action is not linearizable. 
Examples of this arise from conic bundles $X\subset \bA^2\times \bP^2$ given by
$$
t_1t_2=f(x_1,x_2,x_3), 
$$
where $f$ is a form of degree $4$ defining a smooth non-hyperelliptic curve $C\subset\bP^2$, and $\iota$ 
acts by switching $t_1$ and $t_2$. 
This situation arises also from $2$-nodal cubic threefolds, see \cite[Theorems 2.3 and 3.3]{CTZ}.  

We claim that the action of $\iota$ on $\IJ(X)$ is by $(-1)$. There are two families $\{\ell_c\}, \{\ell'_c\}$ of vertical lines in the conic bundle, each parametrized by $C$. The Abel-Jacobi map for such a family of lines defines a non-equivariant isomorphism
\[
\JJ(C) \cong \IJ(X).
\]
The two families are swapped by the $G$-action.
The evident rational equivalence
\[
\ell_{c_0} + \ell'_{c_0} \sim_{\mathrm{rat}}\ell_{c_1} + \ell'_{c_1}, 
\]
for $c_0$, $c_1\in C$,
justifies the claim.
\end{exam}

\subsection*{Actions of $\bZ/3\bZ$}

Our first application is a down-to-earth version of \cite[Example 3.10]{CKK}. 

\begin{exam} 
\label{exam:kk}
Let $X\subset \bA^4$ be given by
$$
x_1x_2x_3=P(x_4), 
$$
where $P$ is a general polynomial of degree $3d$ with $d\ge 2$, with an action of $G=\bZ/3\bZ$ permuting the first $3$ variables. By embedding $X$ in
$$
\mathbb P^1_{(s_1:t_1)} \times \mathbb P^1_{(s_2:t_2)} \times \mathbb P^1_{(s_3:t_3)} \times \mathbb A^1_{x_4},
$$
with $x_i=s_i/t_i$, for $i=1,2,3$, we obtain the defining equation
\[
s_1s_2s_3=P(x_4)t_1t_2t_3
\] 
of a fibration in degree $6$ del Pezzo surfaces, with $3$ singular points of type $\mathsf A_1$ in the fiber over each zero of $P$ for a total of $9d$ singular points.
Compactification in the $\bP^1\times\bP^1\times \bP^1$-fibration
\[
\bP(\cO_{\bP^1}(d)\oplus \cO_{\bP^1})
\times_{\bP^1}\bP(\cO_{\bP^1}(d)\oplus \cO_{\bP^1})
\times_{\bP^1}\bP(\cO_{\bP^1}(d)\oplus \cO_{\bP^1})
\]
over $\bP^1$ does not introduce any further singularities.
By blowing up the singular points we get a smooth projective model $\widetilde{X}$ with
\[ \pi\colon \widetilde{X}\to \bP^1. \]

By Proposition \ref{prop:G} the $G$-action
is not linearizable,
since $G$ fixes a curve of genus $3d-2$, and
$\IJ(\widetilde{X})$ is trivial; the last fact makes no reference to the $G$ action and may be explained either directly or via monodromy.
Directly, $\widetilde{X}$ can be obtained (non-equivariantly) by repeatedly blowing up rational curves in
a $\bP^2$-bundle over $\bP^1$; in doing so the intermediate Jacobian starts out and remains trivial.
But also there is the trivial monodromy of $\pi$, from which it is possible to conclude that $\IJ(\widetilde{X})$ is trivial
using
\cite[Corollary 4.3]{Kanev}.
\end{exam}

\begin{exam}
\label{exam:3node}
Let
\[
X\subset \bP^4_{(x_1:x_2:x_3:x_4:x_5)}
\]
be the $3$-nodal cubic threefold given by 
$$
x_1x_2x_3 +(x_1+x_2+x_3)x_4x_5 + f_3(x_4,x_5)=0,
$$
where $f_3(x_4,x_5)$
is
\begin{enumerate} 
\item 
$\lambda(x_4+x_5)^3$, or
\item $\lambda(x_4+x_5)(x_4-x_5)^2$,  or
\item 
$(x_4 +x_5)(\lambda x_4 +\mu x_5)(\mu x_4 +\lambda x_5)$,
\end{enumerate}
with $\lambda,\mu\in k^\times$
($\lambda^2\ne -1/16$ in (1),
$\lambda^2\ne -27/64$ in (2),
$(\lambda+\mu)^4\ne -1$ and
$(\lambda-\mu)^6\ne 27\lambda\mu$ in (3)).
The intermediate Jacobian of a minimal resolution $\widetilde{X}$ of $X$ is the Jacobian of a genus $2$ curve.
Indeed, projection from two of the nodes expresses $X$ as a conic bundle over $\bP^2$, with a quartic curve as degeneracy locus and equation in split form as in Example \ref{exam:incomp}, see also \cite[Section 3]{CTZ}.
The quartic curve has exactly one node and thus geometric genus $2$.

Let $G=\bZ/3\bZ$ act by permuting the first three variables.
As explained in \cite[Section 4]{CTZ}, 
an equivariant birational model of $\widetilde{X}$ is a
fibration
$$
\pi\colon Y\to \bP^1_{(x_4:x_5)},
$$
with generic fiber a del Pezzo surface of degree 6; it is given by
$$
(x_1:x_2:x_3:x_4:x_5)\mapsto (x_4:x_5). 
$$
The model $Y$ is obtained from $\widetilde{X}$ by performing flops.
Thus
\[ \IJ(Y)\cong \IJ(\widetilde{X}). \] 

Since $G$ does not act on $x_4,x_5$, the $G$-action is trivial on the base $\bP^1$. The $G$-fixed locus on the model $Y$ is an elliptic curve.  
The generic fiber of $\pi$ has Picard rank $3$; it admits three conic fibrations.
The monodromy factors through the $\bZ/2\bZ$, exchanging opposite pairs of lines,
and determines a double cover $C$ of $\bP^1$, branched over $6$ points.
The relative Fano variety of lines is a union of three copies of $C$, permuted by $G$.
The variety of vertical conics is a union of three $\bP^1$-bundles over $\bP^1$,
permuted by $G$.
The variety of vertical rational cubics is a $\bP^2$-bundle over $C$,
where $G$ acts trivially on $C$.

Since the class of vertical rational cubics is $G$-invariant,
it defines the equivariant Abel-Jacobi map
\[ \JJ(C)\to \IJ(Y). \]
We claim that this is surjective, so that the triviality of the $G$-action on the intermediate Jacobian is a consequence of
the trivial $G$-action on $C$.
To see this, we follow the proof of \cite[Corollary 4.3]{Kanev},
but instead of using the relative Fano variety of lines we use
a family of rational cubics parametrized by $C$,
plus conics parametrized by three copies of $\bP^1$.
Such families may be obtained by choosing a section of the
$\bP^2$-bundle over $C$, respectively, the $\bP^1$-bundles over $\bP^1$.

We have observed that $\IJ(Y)$ is the Jacobian of a genus $2$ curve; to this we apply Proposition~\ref{prop:G} and obtain 
$$
I_1=I_2=0, \quad I_3=1.
$$
We conclude that the $G$-action on $X$ is not linearizable.
(Though not needed here, we record the fact $\IJ(Y)\cong \mathrm{J}(C)$, obtained by the analysis of \cite[Section 5]{Kanev}, with $q=3$ in the notation of loc.\ cit.)
\end{exam} 

Example \ref{exam:3node} strengthens \cite[Proposition 4.3]{CTZ}, which showed nonlinearizability of a {\em very general} member of the third family, via specialization.

\bibliographystyle{alpha} 
\bibliography{BurnHodge}

\end{document}